%% file: manuscript.tex
\documentclass[11pt]{amsart}
\input{source/preamble}
\title[Pivotal Grothendieck rings of rank four]{Classifying Grothendieck rings of pivotal fusion categories of rank four}
\author{Jingcheng Dong}
\address{College of Mathematics and Statistics, and Center for Applied Mathematics of Jiangsu Province, Nanjing University of Information Science and Technology, Nanjing 210044, China}
\email{jcdong@nuist.edu.cn}
\author{S\'ebastien Palcoux}
\address{Beijing Institute of Mathematical Sciences and Applications, Huairou District, Beijing 101408, China}
\email{sebastienpalcoux@gmail.com}
\author{Arnaud Plessis}
\address{Beijing Institute of Mathematical Sciences and Applications, Huairou District, Beijing 101408, China}
\email{plessis@bimsa.cn}
\date{}
\begin{document}
\begin{abstract}
We classify the Grothendieck rings of pivotal fusion categories of rank
four over the complex field. There are exactly fifteen, each admitting
a unitary categorification. Central induction and Frobenius--Schur
indicators give a uniform Frobenius--Perron dimension bound of $3600$.
The finite classification combines new arithmetic and twist obstructions
with an exhaustive census and exact exclusion certificates.
\end{abstract}
\maketitle
\setcounter{tocdepth}{2}
\tableofcontents
\newpage
\input{source/introduction}
\input{source/preliminaries}
\input{source/central}
\input{source/uniform}
\input{source/criteria}
\input{source/quadratic}
\input{source/completion}
\input{source/future_work}
\appendix
\input{source/census_appendix}
\input{source/closing}
\clearpage
\input{source/bibliography}

\end{document}

%% file: source/preamble.tex
\usepackage[T1]{fontenc}
\usepackage{lmodern,amsmath,amssymb,mathtools,booktabs,array,microtype}
\usepackage[margin=1in]{geometry}
\usepackage[hidelinks]{hyperref}
\usepackage{enumitem}
\usepackage[section]{placeins}
\setlist[enumerate]{leftmargin=*,itemsep=2pt,topsep=4pt}
\newtheorem{theorem}{Theorem}[section]
\newtheorem{proposition}[theorem]{Proposition}
\newtheorem{lemma}[theorem]{Lemma}
\newtheorem{corollary}[theorem]{Corollary}
\theoremstyle{definition}

\DeclareMathOperator{\FPdim}{FPdim}

\DeclareMathOperator{\Rep}{Rep}
\DeclareMathOperator{\Tr}{Tr}

\DeclareMathOperator{\Ad}{Ad}
\newcommand{\C}{\mathbb C}
\newcommand{\Q}{\mathbb Q}
\newcommand{\Z}{\mathbb Z}
\newcommand{\one}{\mathbf1}
\newcommand{\Fib}{\mathrm{Fib}}
\newcommand{\Near}{\mathrm{NG}}
\numberwithin{equation}{section}
\allowdisplaybreaks
\hypersetup{pdftitle={Classifying Grothendieck rings of pivotal fusion categories of rank four},pdfauthor={Jingcheng Dong, Sebastien Palcoux, and Arnaud Plessis}}

%% file: source/introduction.tex
\section{Introduction}
Fusion categories over the complex field connect representation theory
and subfactor theory with low-dimensional topology and mathematical
physics. Spherical fusion categories furnish state-sum invariants of
three-manifolds \cite{BarrettWestbury}, and unitary fusion categories
provide the algebraic input for string-net models of topological phases
\cite{LevinWen}. Their classification has a strong arithmetic aspect:
dimensions and formal codegrees are algebraic integers constrained by
Galois theory and divisibility \cite{ENO,OstrikCodegrees}. In this paper,
these constraints interact with Frobenius--Schur indicators to bound
dimensions before any finite enumeration is undertaken.
We refer to \cite{EGNO} for the foundations and standard conventions of
fusion category theory.

The \emph{rank} of a fusion category counts the isomorphism classes of
its simple objects; its \emph{multiplicity} is the largest coefficient
in the decomposition of a tensor product of two simple objects.
Fixing the rank alone does not bound the multiplicity of an abstract
fusion ring. A finite classification by rank therefore requires
constraints coming from categorifiability.
In rank four, we obtain an explicit bound from the interaction between
characters of the Grothendieck ring and indicators in the Drinfeld center.
It leads to the following complete classification.

\begin{theorem}\label{thm:main}
Up to fusion-ring isomorphism, the Grothendieck rings of pivotal fusion
categories of rank four over $\C$ are precisely the fifteen fusion rings in
Table~\ref{tab:models}. Each admits a unitary categorification.
\end{theorem}

The theorem classifies fusion rules, with all categorifications of the
same fusion ring counted once. In particular, existence of a
pivotal categorification is equivalent, in rank four, to existence of
a unitary categorification. As a consequence, the sharp global Frobenius--Perron dimension
and multiplicity bounds are $24+12\sqrt3$ and $6$, respectively
(Corollary~\ref{cor:sharp}). These bounds are deduced only after the
classification; the finite search uses the independently proved bound
$\FPdim(\mathcal C)<3600$.

\input{source/models}

Ostrik classified fusion categories of rank two \cite{OstrikRank2}
and pivotal fusion categories of rank three \cite{OstrikRank3}.
In rank four, the modular and premodular cases were studied in
\cite{RSW,HongRowell,Bruillard4}. Without braiding, Larson
\cite{Larson} reduced the non-self-dual pseudo-unitary case to seven
fusion rings. Dong--Zhang--Dai \cite{DongZhangDai} treated the
nontrivially graded self-dual spherical case, and
Edie-Michell--Izumi--Penneys \cite{EIP} classified
$\Z/2\Z$-quadratic unitary fusion categories. Further classifications with
proper fusion subcategories appear in
\cite{DongChenWang,DongSelfDual4}. These results provide several of
the restrictions and models used below; Table~\ref{tab:models} records
the construction references for all fifteen fusion rings.

The proof begins with the spherical-character tests of
Section~\ref{sec:spherical} and the central-induction identities of
Section~\ref{sec:central}. Theorem~\ref{thm:uniform} reduces the problem
to $\FPdim(\mathcal C)<3600$: the non-self-dual case follows from Larson, and an arithmetic
argument controls self-dual pivotal fusion categories outside the pseudo-unitary
Galois orbits. In the remaining case, indicator polynomials concentrate
a hypothetical large dimension in one simple object. Its fusion matrix is
singular; the degree of the dimension field, a matrix argument, and a
cubic trace contradiction rule it out.

The complete census below the bound has $8977$ fusion rings.
Section~\ref{sec:established} complements published structural restrictions
with a new twist obstruction for a quadratic family, including a separate
root-of-unity argument for its boundary case. The remaining candidates
are tested by spherical-character arithmetic and, when Galois reduction
applies, positive central-induction inequalities. Exact certificates
cover every required character orbit and induction-row choice.
Table~\ref{tab:filters} records the successive reductions: $8962$ fusion rings
are excluded, and the fifteen survivors are identified with
Table~\ref{tab:models}. Appendix~\ref{app:census} proves coverage of the
census, including every vanishing-coefficient branch; the supplement
supplies the certificates and separate computational checks, described in
Section~\ref{sec:completion}.

%% file: source/models.tex
% Generated from supplementary/model_table.json; do not edit.
\begin{table}[ht]
\centering\small
\begin{tabular}{@{}clcrl@{}}
\toprule
$\mu(R)$ & A unitary realization & $\FPdim(\mathcal C)$ & Approx. & Reference \\
\midrule
$1$ & $\mathrm{Vec}_{C_2\times C_2}$ & $4$ & $4.000$ & \cite[\S3.4]{LPR} \\
$1$ & $\mathrm{Vec}_{C_4}$ & $4$ & $4.000$ & \cite[\S3.4]{LPR} \\
$1$ & Tambara--Yamagami on $C_3$ & $6$ & $6.000$ & \cite{TambaraYamagami} \\
$1$ & $\Fib\boxtimes\mathrm{Vec}_{C_2}$ & $5+\sqrt{5}$ & $7.236$ & \cite[\S3.4]{LPR} \\
$1$ & $\Rep(D_5)$ & $10$ & $10.000$ & \cite[\S3.4]{LPR} \\
$1$ & $\Fib\boxtimes\Fib$ & $(15+5\sqrt{5})/2$ & $13.090$ & \cite[\S3.4]{LPR} \\
$1$ & $\mathrm{PSU}(2)_6$ & $8+4\sqrt{2}$ & $13.657$ & \cite[Theorem 3.1]{EIP} \\
$1$ & Principal even part of $S'$ & $8+4\sqrt{2}$ & $13.657$ & \cite[Theorem 1]{LMP} \\
$1$ & $\mathrm{PSU}(2)_7$ & $\delta_7$ & $19.234$ & \cite[\S3.4]{LPR} \\
$2$ & $\Rep(A_4)$ & $12$ & $12.000$ & \cite[Theorem 1.1]{Larson} \\
$2$ & $\mathcal H_1$ (Haagerup) & $(39+9\sqrt{13})/2$ & $35.725$ & \cite[Table 2]{GrossmanSnyder} \\
$2$ & Principal even part of $2D2$ & $20+8\sqrt{5}$ & $37.889$ & \cite[Theorem 3.1]{EIP} \\
$3$ & Near-group type $(C_3,3)$ & $(21+3\sqrt{21})/2$ & $17.374$ & \cite[Example 9.4]{IzumiNear} \\
$3$ & $\Ad(E_8)$ & $\delta_E$ & $35.238$ & \cite{IzumiE8,EvansJones} \\
$6$ & Near-group type $(C_3,6)$ & $24+12\sqrt{3}$ & $44.785$ & \cite[Theorem 10.19]{IzumiNear} \\
\bottomrule
\end{tabular}
\caption{One unitary realization of each surviving fusion ring, with its
multiplicity and global Frobenius--Perron dimension
(rounded to three decimals in the fourth column). Rows are ordered first
by multiplicity and then by increasing global Frobenius--Perron dimension.
Here $D_5$ is the dihedral group of order ten, and $\mathrm{PSU}(2)_n$
denotes the adjoint fusion subcategory of $\mathrm{SU}(2)_n$.
The subfactors $S'$ and $2D2$ have indices $3+2\sqrt2$ and $3+\sqrt5$, respectively \cite{LMP,MorrisonPenneys2D2}.
The numbers $\delta_7$ and $\delta_E$ are the largest real roots of
$x^3-27x^2+162x-243$ and
$x^4-45x^3+375x^2-1125x+1125$, respectively.}
\label{tab:models}
\end{table}

%% file: source/preliminaries.tex
\section{Characters and pivotal structures}\label{sec:spherical}
All fusion categories in this section are over $\C$; the rank is arbitrary
unless specified. For a fusion category $\mathcal C$, its Grothendieck
ring $K_0(\mathcal C)$ is the free abelian group on the isomorphism
classes of simple objects, with multiplication induced by $\otimes$.
A fusion ring $R$ is an associative unital ring, free of finite rank $r$
over $\Z$, with distinguished basis $\{b_0=1,b_1,\ldots,b_{r-1}\}$,
nonnegative integral structure constants, and a basis-preserving
anti-involution satisfying
\[
 b_ib_j=\sum_kN_{ij}^kb_k,\qquad
 N_{ij}^0=\delta_{i,j^*},\qquad
 N_{ij}^k=N_{i^*k}^j=N_{kj^*}^i.
\]
The Frobenius--Perron dimension is the unique homomorphism
$\FPdim:R\to\mathbb R$ positive on the distinguished basis;
$d_i:=\FPdim(b_i)$ is the spectral radius of
$N_i=(N_{ij}^k)_{j,k}$ \cite[Theorem~8.2 and Lemma~8.3]{ENO}.
Thus $N_i^{\mathsf T}=N_{i^*}$. We write
\[
 \FPdim(R)=\sum_i d_i^2,\qquad
 \FPdim(\mathcal C)=\FPdim(K_0(\mathcal C)),\qquad
 \mu(R)=\max_{i,j,k}N_{ij}^k.
\]

\begin{lemma}\label{lem:commutative}
Let $R$ be a commutative fusion ring of rank $r$. Its fusion matrices are
normal, with $\|N_i\|=d_i$. The eigenvalues of its symmetric
\emph{Casimir matrix}
\begin{equation}\label{eq:casimir}
 B=\sum_iN_iN_i^{\mathsf T}
\end{equation}
are the formal codegrees
$f_\chi=\sum_i|\chi(b_i)|^2$, one for each character $\chi$ of $R$.
Its largest eigenvalue is $\FPdim(R)$, and
$\mu(R)^2\leq\FPdim(R)$. Every rank-four fusion ring is commutative,
so these conclusions apply in particular at rank four.
\end{lemma}
\begin{proof}
The transposed multiplication matrices give a faithful regular
representation of $R\otimes_\Z\C$ as a finite-dimensional matrix
$*$-algebra, hence a semisimple algebra.
In the commutative case the $N_i$ commute with one another and their
adjoints, so they admit simultaneous unitary diagonalization.
For each character $\chi$, the vector $v_\chi=(\chi(b_k))_k$ satisfies
$N_iv_\chi=\chi(b_i)v_\chi$. Consequently
$\chi(b_{i^*})=\overline{\chi(b_i)}$ and
$Bv_\chi=f_\chi v_\chi$. These $r$ common eigenvectors give the entire
spectrum. Normality gives $\|N_i\|=d_i$, whence
$f_\chi\leq\sum_i d_i^2$, with equality for $\chi=\FPdim$.
Moreover $N_{ij}^k\leq\|N_i\|=d_i\leq\sqrt{\FPdim(R)}$.
Finally, in rank four a noncommutative simple summand would have
complex dimension at least four, in addition to the one-dimensional
Frobenius--Perron representation; semisimplicity rules this out.
\end{proof}
In fact every fusion ring of rank at most five is commutative
\cite[Proposition~9.2]{IntegralRings}.

For a commutative fusion ring, write
$\chi_0=\FPdim,\chi_1,\ldots,\chi_{r-1}$ and $f_j=f_{\chi_j}$.
The formal codegrees are positive algebraic integers, with
$f_0=\FPdim(R)$. The normalized vectors $v_{\chi_j}/\sqrt{f_j}$
form an orthonormal basis, giving
\begin{equation}\label{eq:orthogonality}
 \sum_{j=0}^{r-1}\frac{\chi_j(b_i)\overline{\chi_j(b_k)}}{f_j}
 =\delta_{ik},\qquad \sum_{j=0}^{r-1}f_j^{-1}=1.
\end{equation}
The second equality, obtained by taking $i=k=0$, is the
\emph{reciprocal formal-codegree identity}. In particular $f_j>1$
if $r\geq2$. For a self-dual basis all characters are real-valued.

A pivotal structure on $\mathcal C$ is a monoidal natural isomorphism
$j:\operatorname{Id}_{\mathcal C}\Rightarrow(-)^{**}$
\cite[Definition~4.7.7]{EGNO}. It defines a dimension character
$\chi([V])=\dim_{\mathcal C}(V)$, nonzero on every simple object
\cite[Propositions~4.7.12 and~4.8.4]{EGNO}. It is spherical exactly when
$\dim_{\mathcal C}(V)=\dim_{\mathcal C}(V^*)$ for every simple $V$
\cite[Definition~4.7.14]{EGNO}; for a commutative Grothendieck ring this
is equivalent to these dimensions being real, by the character identity above.
For a fusion category $\mathcal D$, write $\mathcal O(\mathcal D)$ for a
set of representatives of its simple isomorphism classes.
The categorical dimension $\dim(\mathcal C)$ is defined without a
pivotal structure as the sum of squared norms of simple objects
\cite[Definition~7.21.3]{EGNO}; in the pivotal case it equals
$\sum_{V\in\mathcal O(\mathcal C)}\dim_{\mathcal C}(V)\dim_{\mathcal C}(V^*)$.

\begin{lemma}\label{lem:spherical}
Every pivotal fusion category of rank four admits a spherical structure.
More precisely, every pivotal structure on a nonpointed fusion category
of rank four is spherical. Thus pivotal and spherical categorifiability
give the same rank-four fusion rings.
\end{lemma}
\begin{proof}
In a pointed fusion category the dimensions of invertible simple objects
define a character of their group. Multiplying each simple component $j_V$ by
$\dim_{\mathcal C}(V)^{-1}$ therefore gives another pivotal structure,
whose simple dimensions are all one; it is spherical.
Now assume that $\mathcal C$ is nonpointed. If every simple is self-dual,
the given pivotal structure is spherical by definition. Otherwise write
the simple classes as $1,x,y,z$, with $x^*=z$ and $y^*=y$, and let
$\chi$ be the given dimension character.
The scalars $s_b=\chi(b)/\chi(b^*)$ are the components of the
\emph{slope}, a monoidal natural automorphism of the identity
\cite[Section~4C]{BVCenter}. Thus $N_{ab}^c>0$ implies $s_c=s_as_b$.
Their values are $1,s,1,s^{-1}$. Suppose $s\ne1$.

If $s^2\ne1$, every basis summand of $x^2$ has slope $s^2$, which is
neither $1$ nor $s$. Since $x^2\ne0$, necessarily $s^2=s^{-1}$ and
$x^2=cz$ for an integer $c\geq1$. Similarly, $xy=qx$ with an integer
$q\geq1$, and Frobenius reciprocity gives $xz=1+qy$.
Writing $d=\FPdim(x)=\FPdim(z)$, these identities give
$d=c$, $\FPdim(y)=q$, and $c^2=1+q^2$.
Thus $(c-q)(c+q)=1$, impossible for positive integers $c,q$.

If $s=-1$, the slope-one summands of $x^2$ can only be $1,y$,
and $[x^2:1]=0$. Hence $x^2=\ell y$ and $xz=1+qy$, with
$\ell\geq1$, $q\geq0$. Taking Frobenius--Perron dimensions gives
$(\ell-q)\FPdim(y)=1$. Since $\ell-q$ is a positive integer and
$\FPdim(y)\geq1$, both factors are one. Therefore $y$ is invertible
and multiplication by $y$ permutes the distinguished basis.
Frobenius reciprocity now gives $\ell=[xy:z]\leq1$, so
$\ell=1$, $q=0$, and $\FPdim(x)=\FPdim(z)=1$.
This makes $\mathcal C$ pointed, a contradiction. All slopes are
therefore one, proving sphericality.
\end{proof}

A fusion category is pseudo-unitary if
$\dim(\mathcal C)=\FPdim(\mathcal C)$; it then has a canonical spherical
structure with simple dimensions $d_i$ \cite[Proposition~9.5.1]{EGNO}.
Every fusion category with $\FPdim(\mathcal C)\in\Z$ is pseudo-unitary
\cite[Proposition~9.6.5]{EGNO}.

\medskip\noindent\textbf{Spherical-character arithmetic.}\label{sec:spherical-arithmetic}
Here $\mathcal C$ is spherical of arbitrary rank $r$, with commutative
Grothendieck ring and simples $X_0=\one,\ldots,X_{r-1}$.
Put $c_i=\dim_{\mathcal C}(X_i)$ and $K_c=\Q(c_0,\ldots,c_{r-1})$.
The character $c$ is nonzero on every simple, and $c_i=c_{i^*}$.
For any embedding $\sigma:K_c\hookrightarrow\C$, the values
$\sigma(c_i)$ still define a character, since the fusion coefficients
are integers. Applying the character identity to this new character
gives $\sigma(c_i)=\sigma(c_{i^*})=\overline{\sigma(c_i)}$.
Thus $K_c$ is totally real, and $\dim(\mathcal C)=\sum_i c_i^2$.
Every formal codegree $f$ lies in $K_c$, and
$\dim(\mathcal C)/f$ is an algebraic integer
\cite[Theorem~2.13 and Corollaries~2.14--2.15]{OstrikRank3}.
In particular, the characteristic polynomial of $B$ must split over $K_c$.
Ostrik's spherical inequality
\cite[Theorem~3.1.1 and Remark~3.1.2]{OstrikGlobal} gives
\begin{equation}\label{eq:spherical-mass}
 w:=\sum_jf_j^{-2}=\Tr(B^{-2})\leq\frac12+\frac1{2\dim(\mathcal C)}.
\end{equation}
The left side is rational. Applying the inequality to Galois conjugates
of the spherical structure shows that it holds with
$\dim(\mathcal C)$ replaced by any of its conjugates.

Divisibility also gives an elementary norm condition. For nonzero
algebraic integers $u,v$, write $e_u=[\Q(u):\Q]$ and
$n_u=|N_{\Q(u)/\Q}(u)|$, and similarly for $v$. Then
\begin{equation}\label{eq:norm-test}
 \frac uv\text{ integral}\quad\Longrightarrow\quad
 \frac{v_p(n_u)}{e_u}\geq\frac{v_p(n_v)}{e_v}
 \quad\text{for every prime }p.
\end{equation}
Here $v_p(p)=1$. Indeed, for $K=\Q(u,v)$ the norm of $u/v$ is a
nonzero integer, and its $p$-adic valuation is
$[K:\Q]\bigl(v_p(n_u)/e_u-v_p(n_v)/e_v\bigr)$.

\begin{lemma}\label{lem:residue-field}
Let $F\in\Z[t]$ be monic and irreducible, $K=\Q[t]/(F)$, and
$P\in\Z[t]$ monic. If, for some prime $p$, the reduction of $F$ has
a simple root in $\mathbb F_p$ but the reduction of $P$ has an
irreducible factor of degree greater than one, then $P$ does not split
over $K$.
\end{lemma}
\begin{proof}
Hensel's lemma lifts the simple root to a root of $F$ in $\Q_p$,
giving an embedding $K\hookrightarrow\Q_p$. If $P$ split over $K$,
its roots would be $p$-adic integers, since $P$ is monic. Reduction
would then split $P$ into linear factors over $\mathbb F_p$, a
contradiction.
\end{proof}
We apply this lemma with $K=K_c$ and $P$ the characteristic polynomial
of $B$ to rule out candidate spherical dimension characters whose
fields do not contain all formal codegrees.

\begin{lemma}\label{lem:galois-reduction}
Let $\mathcal C$ be a fusion category of arbitrary rank over $\C$.
If $\dim(\mathcal C)$ is Galois conjugate to $\FPdim(\mathcal C)$,
then a Galois conjugate of $\mathcal C$ is pseudo-unitary and has the
same fusion ring.
\end{lemma}
\begin{proof}
Extend an automorphism of $\overline{\Q}$ sending
$\dim(\mathcal C)$ to $\FPdim(\mathcal C)$ to an automorphism of $\C$,
and apply it to the associativity and rigidity data. The resulting
fusion category has categorical dimension $\FPdim(\mathcal C)$ and
unchanged fusion coefficients, so it is pseudo-unitary. Equip it with
its canonical spherical structure \cite[Proposition~9.5.1]{EGNO}.
\end{proof}
Consequently, pseudo-unitary obstructions exclude a candidate spherical
dimension character whenever its formal codegree $\dim(\mathcal C)$
has the same minimal polynomial as $\FPdim(\mathcal C)$. This condition
concerns the two global dimensions; the dimension characters themselves
need not be Galois conjugate. To exclude spherical categorifiability,
every remaining Galois orbit of candidate dimension characters must
also be ruled out.

%% file: source/central.tex
\section{Central induction and indicator inequalities}\label{sec:central}
Initially, let $\mathcal C$ be a spherical fusion category over $\C$ of
arbitrary rank. For an object $U$ of a fusion category $\mathcal D$,
write $\mathcal O(U)=\{A\in\mathcal O(\mathcal D):[U:A]>0\}$ for its
simple support. Let
$F:\mathcal Z(\mathcal C)\to\mathcal C$ be the forgetful functor from the
Drinfeld center, and let $I$ be its right adjoint. The twist acts on a
simple $A\in\mathcal O(\mathcal Z(\mathcal C))$ by a root of unity
$\theta_A$ \cite[Theorem~5.5]{NgSchauenburg}. The higher Frobenius--Schur indicators satisfy
\begin{equation}\label{eq:indicator-trace}
 \sum_{A\in\mathcal O(\mathcal Z(\mathcal C))}
 [I(V):A]\dim_{\mathcal Z(\mathcal C)}(A)\theta_A^n
 =\dim(\mathcal C)\nu_n(V)\qquad(n\geq1)
\end{equation}
for every $V\in\mathcal C$ \cite[Theorem~4.1]{NgSchauenburg}.
Here $\nu_1(V)=[V:1]$; for simple $X$, $\nu_2(X)=0$ if
$X\not\cong X^*$, and $\nu_2(X)\in\{-1,1\}$ otherwise.
Moreover, $\theta_A=1$ for every $A\in\mathcal O(I(1))$
\cite[Theorem~2.5]{OstrikRank3}.

For the rest of this section, assume that $\mathcal C$ is pseudo-unitary,
with its canonical spherical structure. Thus categorical dimensions
are Frobenius--Perron dimensions. This positivity assumption is essential;
we apply these tests to pivotal categorifications only after
Lemma~\ref{lem:galois-reduction}, or directly when
$\FPdim(\mathcal C)\in\Z$ by \cite[Proposition~8.24]{ENO}.
Throughout this section, $\zeta$ ranges over roots of unity, as do all
twist eigenvalues; in particular $|\zeta|=1$. Define the \emph{low part}
and the \emph{twist sectors} by
\[
 I_{\mathrm{low}}(V)=\bigoplus_{A\in\mathcal O(I(1))}A^{\oplus[I(V):A]},
 \qquad
 I(V)_\zeta=\bigoplus_{\substack{A\in\mathcal O(\mathcal Z(\mathcal C))\\
                          \theta_A=\zeta}}A^{\oplus[I(V):A]}.
\]
The low part retains multiplicities from $I(V)$, irrespective of those
in $I(1)$; it is a direct summand of $I(V)_1$.

\begin{lemma}[Polynomial indicator bounds]\label{lem:indicator-polynomials}
Let $V\in\mathcal C$ and let
\[
 P(z)=c_0+\sum_{k=1}^s(c_kz^k+\overline{c_k}z^{-k}),\qquad
 c_0\in\mathbb R,\quad c_k\in\C,
\]
be nonnegative on the unit circle. Then
\begin{equation}\label{eq:polynomial-low}
 P(1)\FPdim(I_{\mathrm{low}}(V))
 \leq\FPdim(\mathcal C)
 \left(c_0\FPdim(V)+2\operatorname{Re}\sum_{k=1}^s c_k\nu_k(V)\right).
\end{equation}
For every twist sector, one also has
\begin{equation}\label{eq:polynomial-sector}
 P(1)\FPdim(I(V)_\zeta)
 \leq\FPdim(\mathcal C)
 \left(c_0\FPdim(V)+2\operatorname{Re}
                   \sum_{k=1}^s c_k\zeta^{-k}\nu_k(V)\right).
\end{equation}
There is no commutativity assumption on $K_0(\mathcal C)$.
\end{lemma}
\begin{proof}
In a sum over central simples $A$, the nonnegative coefficient
$[I(V):A]\FPdim(A)$ is the dimension of the $A$-isotypic summand
of $I(V)$. Its sum is
$\FPdim(I(V))=\FPdim(\mathcal C)\FPdim(V)$, by
$FI(V)\cong\bigoplus_{X\in\mathcal O(\mathcal C)}X\otimes V\otimes X^*$
\cite[Proposition~2.4]{OstrikRank3}.
Its $k$th twist moment, obtained by multiplying each coefficient by
$\theta_A^k$, is \eqref{eq:indicator-trace}; the negative moments are
complex conjugates because these coefficients are real and
$|\theta_A|=1$. Consequently
\[
 \sum_A[I(V):A]\FPdim(A)P(\theta_A)
 =\FPdim(\mathcal C)
  \left(c_0\FPdim(V)+2\operatorname{Re}\sum_{k=1}^s c_k\nu_k(V)\right).
\]
Retaining the terms with $A\in\mathcal O(I(1))$, all of whose twists
are one, proves \eqref{eq:polynomial-low}. Replacing $P(z)$ by
$P(\zeta^{-1}z)$ and retaining the terms with $\theta_A=\zeta$ proves
\eqref{eq:polynomial-sector}.
\end{proof}

\smallskip\noindent\textbf{Higher-indicator bounds.} If $\nu_1(V)=\cdots=\nu_m(V)=0$ for an integer $m\geq1$,
apply the lemma to
\[
 |1+z+\cdots+z^m|^2
 =(m+1)+\sum_{k=1}^m(m+1-k)(z^k+z^{-k})
\]
to obtain, for every twist sector,
\begin{equation}\label{eq:vanishing-indicators}
 \begin{split}
 (m+1)\FPdim(I_{\mathrm{low}}(V))&\leq\FPdim(\mathcal C)\FPdim(V),\\
 (m+1)\FPdim(I(V)_\zeta)&\leq\FPdim(\mathcal C)\FPdim(V).
 \end{split}
\end{equation}
For $m=1$ the low-part bound
is the estimate in \cite[proof of Theorem~2.21]{OstrikRank3}; the
extension to twist sectors was suggested by Andrew Schopieray.
For a direct sum of non-self-dual simple objects one can take $m=2$.

To express these bounds using fusion data, now assume that
$K_0(\mathcal C)$ is commutative of rank $r$. Order the simples as
$X_0=1,\ldots,X_{r-1}$ and the central simples as
$Z_0=1,\ldots,Z_{t-1}$. Put $d_i=\FPdim(X_i)$ and
$d=(d_i)_{i=0}^{r-1}$. Recall the Casimir matrix
$B=\sum_iN_iN_i^{\mathsf T}$ from \eqref{eq:casimir}.
Define the \emph{induction matrix} by
$M_{ai}=[I(X_i):Z_a]=[F(Z_a):X_i]$, with central simples indexing its
rows. Let $L$ consist of the rows indexed by $\mathcal O(I(1))$.
These are $r$ distinct simple summands, each of multiplicity one, with
dimensions $\FPdim(\mathcal C)/f_j$, one for each character of
$K_0(\mathcal C)$ \cite[Theorem~2.13]{OstrikRank3}.
Adjunction, the formula for $FI$, and commutativity give
\begin{equation}\label{eq:low-gram}
 M^{\mathsf T}M=B,\qquad L^{\mathsf T}\mathbf1=Be_0,\qquad
 Q:=B-L^{\mathsf T}L,
\end{equation}
where $e_0=(1,0,\ldots,0)^{\mathsf T}$, so $Be_0$ is the first
column of $B$. Thus $Q$ is the Gram matrix of the remaining induction
rows. One row of $L$ is $e_0^{\mathsf T}$, corresponding to the central
unit. Writing a low row as $v^{\mathsf T}$, its formal codegree $f$
requires
\begin{equation}\label{eq:low-columns}
 v\in\Z_{\geq0}^r,\qquad v_0=1,\qquad
 d^{\mathsf T}v=\FPdim(\mathcal C)/f,\qquad
 v_i\leq\min\{B_{0i},\lfloor\sqrt{B_{ii}}\rfloor\}.
\end{equation}
The bounds follow from the first column and diagonal entries of
$M^{\mathsf T}M=B$. Furthermore,
\[
 W_i:=\FPdim(I_{\mathrm{low}}(X_i))
 =\sum_{a:\,Z_a\in\mathcal O(I(1))}M_{ai}\sum_jM_{aj}d_j
 =(L^{\mathsf T}Ld)_i.
\]
For a nonunit simple $X_i$, Lemma~\ref{lem:indicator-polynomials}
applied to $|1+z+z^2|^2$ gives
\begin{equation}\label{eq:row-indicator}
 9W_i\leq\FPdim(\mathcal C)\bigl(3d_i+2\nu_2(X_i)\bigr).
\end{equation}
Put $a_i=N_{ii}^{i^*}$, equal to $N_{ii}^{i}$ when $X_i$ is
self-dual. Applying the lemma to
\[
 |(1+z)(1+z+z^2)|^2
 =10+8(z+z^{-1})+4(z^2+z^{-2})+(z^3+z^{-3})
\]
yields
\begin{equation}\label{eq:row-third}
 36W_i\leq\FPdim(\mathcal C)(10d_i+8+2a_i).
\end{equation}
Indeed, $\nu_1(X_i)=0$, $\nu_2(X_i)\leq1$, and
$|\nu_3(X_i)|\leq\dim_{\C}\operatorname{Hom}_{\mathcal C}
(1,X_i^{\otimes3})=a_i$: the third indicator is the trace of an
operator whose cube is the identity
\cite[Theorem~5.1]{NgSchauenburgHigher}.

The scalar inequalities \eqref{eq:row-indicator} and
\eqref{eq:row-third} also hold without commutativity, since their
polynomial proofs do not use the matrix description of $W_i$.

The following \emph{local central induction} test combines these
constraints. It applies in every rank with commutative Grothendieck
ring; below we use it with $r=4$.
\begin{proposition}\label{prop:low-test}
A pseudo-unitary categorification requires a choice of $r$ low
induction rows, one for each formal codegree, satisfying
\eqref{eq:low-columns}, including $e_0^{\mathsf T}$, and obeying
$L^{\mathsf T}\mathbf1=Be_0$, with
\begin{align}
 9W_i&\leq\FPdim(\mathcal C)(3d_i+2)
       &&(i>0,\ i=i^*),\label{eq:row-selfdual}\\
 3W_i&\leq\FPdim(\mathcal C)d_i
       &&(i\ne i^*).\label{eq:row-pair}
\end{align}
If an invertible involution $g\ne1$ satisfies $Q_{gg}=2$,
$\FPdim(\mathcal C)/2-W_g=1$, and $W_g>0$, it also requires
\begin{equation}\label{eq:involution-test}
 2W_i+\FPdim(\mathcal C)Q_{gi}
 \leq\FPdim(\mathcal C)(d_i+1)\qquad(i\ne0,g).
\end{equation}
If every bounded row choice violates one of these necessary conditions,
the fusion ring has no pseudo-unitary categorification. Permutations
among nonunit low rows with equal formal codegrees may be identified.
\end{proposition}
\begin{proof}
The first two inequalities follow from \eqref{eq:row-indicator}.
For the last, put $D=\FPdim(\mathcal C)$. Since
$Q_{gg}=\sum_{Z_a\notin\mathcal O(I(1))}M_{ag}^2=2$,
there are exactly two additional central simples in $I(g)$, each once.
The zeroth twist moment is the sum of the dimensions of the simple
summands, namely $\FPdim(I(g))=D$; the second is
$\sum_aM_{ag}\FPdim(Z_a)\theta_{Z_a}^2=D\nu_2(g)$.
As $|\nu_2(g)|=1$, equality in the triangle inequality forces
$\theta_{Z_a}^2=\nu_2(g)$ whenever $M_{ag}>0$.
The assumption $W_g>0$ supplies such a summand of twist one, so
$\nu_2(g)=1$ and all these twists lie in $\{1,-1\}$.
Their first moment is zero; hence the sums of their dimensions at
these two twists are both $D/2$. The low part has twist one and
dimension $W_g$, leaving additional dimensions $D/2-W_g=1$ at
$1$ and $D/2$ at $-1$. Both are positive, so the two additional
simples have these respective dimensions and twists.
The dimension-one simple restricts to $g$. The other's induction
row is therefore $(Q_{gi}-\delta_{gi})_i$.
For $i\ne0,g$, the latter simple and the low part contribute
$W_i+(D/2)Q_{gi}$ to the dimension supported on $\theta^2=1$
in $I(X_i)$. Apply the nonnegative function
$1+\operatorname{Re}(z^2)$, which equals $2$ when $z^2=1$,
to its zeroth and second moments. This bounds that dimension by
$D(d_i+\nu_2(X_i))/2\leq D(d_i+1)/2$, proving
\eqref{eq:involution-test}.
Finally, the bounds \eqref{eq:low-columns} include every actual low
part. Permuting rows with equal formal codegrees changes none of
$L^{\mathsf T}L$, $Q$, $W$, or the stated conditions.
\end{proof}

%% file: source/uniform.tex
\section{A uniform dimension bound}\label{sec:uniform}
All fusion categories in this section are over $\C$ and have rank four,
unless another rank is stated explicitly. The estimates of
Section~\ref{sec:central} yield a bound independent of
any census. This bound makes the subsequent classification a finite
problem; the sharper bounds in Corollary~\ref{cor:sharp} follow only
after that classification.
\begin{theorem}\label{thm:uniform}
Every rank-four fusion ring $R$ admitting a pivotal categorification satisfies
\[
 \FPdim(R)<3600,\qquad \mu(R)\leq59.
\]
\end{theorem}
\subsection{Reduction to the self-dual pseudo-unitary case}\label{sec:pivotal-reduction}

\begin{lemma}\label{lem:exceptional}
Let $\mathcal C$ be a self-dual pivotal fusion category of rank four. Either it has a pseudo-unitary Galois conjugate, or $\FPdim(\mathcal C)<576$.
\end{lemma}
\begin{proof}
Write $R=K_0(\mathcal C)$ and $D=\FPdim(\mathcal C)$.
Self-duality makes the pivotal structure spherical. If $D$ is rational,
it is an integer and the fusion category is pseudo-unitary by \cite[Proposition~8.24]{ENO}. For the rest of the proof, assume that no Galois conjugate is pseudo-unitary. Then $D$ is irrational, and the spherical dimension character is outside the Frobenius--Perron orbit.

The field generated by spherical dimensions contains every formal codegree, including $D$, by \cite[Corollary~2.15]{OstrikRank3}. Here the \emph{Frobenius--Perron field} means
$K=\Q(\FPdim(X):X\in\mathcal O(\mathcal C))$, the field of
values of the Frobenius--Perron character; it need not equal $\Q(D)$
a priori. Its degree is the size of that character orbit, hence lies
between two and four. Degree four leaves no other character orbit, and in degree three the other orbit is rational, so neither case is possible. If $[K:\Q]=2$, then $\Q(D)=K$ because $D$ is irrational.
The other character field contains $K$, and only two dimensions remain
in $R\otimes\Q$; it must therefore be another copy of $K$. Thus
\begin{equation}\label{eq:repeated-quadratic}
 R\otimes\Q\cong K\times K
\end{equation}
for a real quadratic field $K$, and the spherical dimension character belongs to the other factor.

Denote the four formal codegrees by $D,D',H,H'$, with primes indicating conjugation in $K$. We allow $H=H'$ and list one formal codegree for each character. After Galois conjugation, arrange that the spherical global dimension is $H\geq H'$. Each formal codegree is positive and greater than one. Also $H\leq D$, because each character value has absolute value at most the corresponding Frobenius--Perron dimension. Equality would make the fusion category pseudo-unitary, so $H<D$ in the exceptional case.

By \cite[Theorem~2.13]{OstrikRank3}, $H/D$ is the dimension of a simple object of $\mathcal Z(\mathcal C)$
and hence an algebraic integer. Put
\[
 n=DD',\qquad k=\frac{HH'}{DD'}\in\Z_{\geq1},\qquad
 p=\frac1D+\frac1{D'},\qquad
 w=\frac1{D^2}+\frac1{(D')^2}+\frac1{H^2}+\frac1{(H')^2}.
\]
Here $k$ is the field norm of $H/D$, including when that element is rational. Since $D'>1$ and $H\geq H'$, we have
\begin{equation}\label{eq:exceptional-size}
 n>D,\qquad H\geq\sqrt{kn}>\sqrt{kD}.
\end{equation}
The reciprocal formal-codegree identity $D^{-1}+(D')^{-1}+H^{-1}+(H')^{-1}=1$
from \eqref{eq:orthogonality} gives
\begin{equation}\label{eq:exceptional-w}
 w-\frac12=2\left(p-\frac12\right)^2-\frac{2(1+1/k)}n.
\end{equation}

Ostrik's spherical inequality \eqref{eq:spherical-mass},
applied with global dimension $H$, gives
\[
 w\leq\frac12+\frac1{2H}.
\]

Set $r=D/H>1$. By \eqref{eq:exceptional-size}, $r<\sqrt{D/k}\leq\sqrt D$. Moreover, $kr>2$. This is immediate when $k\geq2$. Both $H/D$ and its conjugate $H'/D'$ are positive. When $k=1$,
$H/D$ is therefore a totally positive algebraic-integer unit of norm one. Thus $r$ is also such a unit, with conjugate $r^{-1}$; the integer $r+r^{-1}>2$ is at least three, and $r\geq(3+\sqrt5)/2>2$.

Since $H'=krD'$, the same identity gives
\[
 1=\frac{1+r}{D}+\frac{1+1/(kr)}{D'},\qquad
 p=\frac{kr}{kr+1}-\frac{kr^2-1}{(kr+1)D}
 >\frac23-\frac rD>\frac23-\frac1{\sqrt D}.
\]
If $D\geq576$, then $p>5/8$ and $H>24$. Equation~\eqref{eq:exceptional-w} consequently gives
\[
 w-\frac12>\frac1{32}-\frac4D
 \geq\frac7{288}>\frac1{48}>\frac1{2H},
\]
contradicting \eqref{eq:spherical-mass}. Thus $D<576$, as required.
\end{proof}

\begin{lemma}\label{lem:nonselfdual}
A pivotal fusion category of rank four with a non-self-dual simple has Frobenius--Perron dimension less than $45$.
\end{lemma}
\begin{proof}
A pointed fusion category has Frobenius--Perron dimension four, so assume
that the fusion category is not pointed. Write the basis as $1,X,Y,Z$, where $X^*=Z$ and $Y^*=Y$.

By Lemma~\ref{lem:spherical}, the given pivotal structure is spherical.

There are exactly two real-valued characters of this fusion ring. Indeed, by the adjoint relation $N_{U^*}=N_U^{\mathsf T}$ and simultaneous unitary diagonalization, $\chi(U^*)=\overline{\chi(U)}$ for every basis element $U$. Extend basis duality $\C$-linearly. This involution has trace two on the distinguished basis; on the primitive idempotents it is a permutation, so exactly two characters are fixed, precisely the real-valued ones. If the Frobenius--Perron character is irrational, its Galois orbit has at least two members, all fixed by duality because
$\FPdim(U)=\FPdim(U^*)$, and hence consists of these two characters. The spherical dimension character is real-valued, so it belongs to that orbit. Lemma~\ref{lem:galois-reduction} gives a pseudo-unitary Galois conjugate. If the Frobenius--Perron character is rational, the global Frobenius--Perron dimension is an integer and pseudo-unitarity follows directly from \cite[Proposition~8.24]{ENO}.

It remains to use the pseudo-unitary candidate theorem \cite[Theorem~1.1]{Larson}. Besides $\Z[C_4]$, the candidates have simple dimensions
\[
 1,1,1,\frac{m+\sqrt{m^2+12}}2
 \quad(m=0,2,3,6),
\]
or
\[
 1,1,c+\sqrt{c^2+1},c+\sqrt{c^2+1}
 \quad(c=1,2).
\]
The first family has total dimension at most $24+12\sqrt3<45$, and the second at most $20+8\sqrt5<38$. Galois conjugation leaves the fusion ring, and hence its Frobenius--Perron dimension, unchanged. This proves the lemma.
\end{proof}

\subsection{Indicator estimates and concentration}\label{sec:concentration}

We first work in arbitrary rank $r\geq2$. Let $\mathcal C$ be a
pseudo-unitary fusion category with simple objects
$1,X_1,\ldots,X_{r-1}$. Use its canonical spherical structure. Put
\[
 d_i=\FPdim(X_i),\qquad S=\sum_{i=1}^{r-1}d_i,\qquad
 a_i=N_{ii}^{i^*},\qquad A=\sum_{i=1}^{r-1}d_i a_i.
\]
For the irreducible representations $E$ of $K_0(\mathcal C)\otimes\C$,
let $f_E$ be the scalar by which
$\sum_i\Tr_E([X_i])[X_i^*]$ acts on $E$, its \emph{formal codegree}
\cite[Section~2.3]{OstrikRank3}, and put $w=\sum_E f_E^{-2}$.
For a commutative fusion ring these are the formal codegrees of its
one-dimensional characters.

\begin{lemma}\label{lem:short-moments}
For every pseudo-unitary fusion category of rank $r\geq2$,
\begin{align}
 9\bigl(\FPdim(\mathcal C)^2w-\FPdim(\mathcal C)\bigr)
 &\leq\FPdim(\mathcal C)\bigl(3(\FPdim(\mathcal C)-1)+2S\bigr),
       \label{eq:short-second-moment}\\
 36\bigl(\FPdim(\mathcal C)^2w-\FPdim(\mathcal C)\bigr)
 &\leq\FPdim(\mathcal C)\bigl(10(\FPdim(\mathcal C)-1)+8S+2A\bigr).
       \label{eq:short-third-moment}
\end{align}
Consequently,
\[
 A\geq\left(\frac{18}{r-1}-5\right)\FPdim(\mathcal C)
 -13-\frac{36}{r-1}-4S+
 \frac{18r}{(r-1)\FPdim(\mathcal C)}.
\]
In particular, for $r=4$,
\begin{equation}\label{eq:short-diagonal-lower}
 A\geq\FPdim(\mathcal C)-25-4S+
 \frac{24}{\FPdim(\mathcal C)}.
\end{equation}
\end{lemma}

\begin{proof}
Put $D=\FPdim(\mathcal C)$ and
$W_i=\FPdim(I_{\mathrm{low}}(X_i))$. By \cite[Theorem~2.13]{OstrikRank3},
the simple summand $A_E$ of $I(1)$ has dimension $u_E=D/f_E$ and
multiplicity $m_E=\dim_{\C}E$. Put $M_{E,i}=[F(A_E):X_i]$.
Its induction row therefore satisfies
$\sum_{i>0}d_iM_{E,i}=u_E-m_E$, and
\[
 \sum_{i>0}d_iW_i=\sum_E u_E(u_E-m_E)=D^2w-D,
\]
since $\sum_E m_Eu_E=\dim_{\mathcal Z(\mathcal C)}I(1)=D$.
Apply Lemma~\ref{lem:indicator-polynomials} to $X_i$ with
$|1+z+z^2|^2$ and $|(1+z)(1+z+z^2)|^2$, respectively.
Using $\nu_2(X_i)\leq1$ and $\operatorname{Re}\nu_3(X_i)\leq a_i$
gives the bounds \eqref{eq:row-indicator} and~\eqref{eq:row-third}
without a commutativity hypothesis. Multiplying by $d_i$ and summing
proves the first two assertions. The second gives $A+4S\geq18Dw-5D-13$.
Finally, $\sum_E m_E^2=r$ and $\sum_E m_E/f_E=1$.
Separating the Frobenius--Perron representation, for which $m_E=1$
and $f_E=D$, weighted Cauchy--Schwarz gives
\[
 w\geq\frac1{D^2}+\frac1{r-1}\left(1-\frac1D\right)^2.
\]
Substitution gives the general lower bound and its specialization
to $r=4$.
\end{proof}

For the rest of this subsection, assume $r=4$ and that every simple
object is self-dual; in particular, $a_i=N_{ii}^{i}$.
The fusion ring is commutative. Index its characters so that
$\chi_0=\FPdim$ and $f_0=\FPdim(\mathcal C)$.
Character orthogonality gives
\begin{equation}\label{eq:short-character-identities}
 a_i=\sum_{j=0}^3\frac{\chi_j(X_i)^3}{f_j},
\end{equation}
by taking the coefficient of $1$ in $X_i^3$.

\begin{lemma}\label{lem:short-concentration}
Suppose that $\FPdim(\mathcal C)\geq3600$. Put
$M=\max_i d_i^2$ and $E=\FPdim(\mathcal C)-M$. Then
\begin{gather}
 2<f_j<4\quad(1\leq j\leq3),\qquad f_1f_2f_3<36,
       \label{eq:short-codegrees}\\
 E<\frac72\sqrt{\FPdim(\mathcal C)}.
       \label{eq:short-dominance}
\end{gather}
There is a unique simple object $X$ of maximum dimension. Its fusion
matrix is singular, and
\begin{equation}\label{eq:short-small-eigenvalues}
 |\chi_j(X)|^2<\frac{14}{\sqrt{\FPdim(\mathcal C)}}
 \leq\frac7{30}<\frac14
 \qquad(1\leq j\leq3).
\end{equation}
\end{lemma}

\begin{proof}
Put $D=\FPdim(\mathcal C)$, $t=\sqrt D\geq60$, $q_j=f_j^{-1}$,
$m=(1-D^{-1})/3$, and $V=\sum_{j=1}^3(q_j-m)^2$.
Substitution of $w=D^{-2}+3m^2+V$ into
\eqref{eq:short-second-moment} gives
\[
 V\leq\frac4{3D}+\frac{2S}{9D}-\frac4{3D^2}
 <\frac4{3t^2}+\frac7{18t}\leq\frac{37}{5400}<\frac1{144},
\]
where $S<\sqrt3t<7t/4$.
Three real numbers of sum zero each have square at most two thirds of
their squared norm. Therefore
\[
 |q_j-m|\leq\sqrt{2V/3}<\frac1{14},\qquad
 \frac14<\frac{11}{42}-\frac1{3D}<q_j
 <\frac{17}{42}<\frac12.
\]
This proves $2<f_j<4$. Writing $q_j=1/4+r_j$ with $r_j>0$ gives
\[
 q_1q_2q_3
 \geq\frac1{64}+\frac{r_1+r_2+r_3}{16}
 =\frac1{32}-\frac1{16D}>\frac1{36},
\]
and hence the product bound.

Relabel so that $d_3^2=M$. Since $f_j<4$, every non-Frobenius--Perron
character value on a nonunit simple has absolute value less than $2$.
Equations~\eqref{eq:orthogonality} and \eqref{eq:short-character-identities} imply
\begin{equation}\label{eq:short-spectral-upper}
 A\leq\frac{\sum_i d_i^4}{D}
    +2\sum_i d_i\left(1-\frac{d_i^2}{D}\right).
\end{equation}
Initially, $\sum_i d_i^4/D\leq M$ and
\eqref{eq:short-diagonal-lower} give
\begin{equation}\label{eq:short-coarse-dominance}
 E\leq25+6S<25+12t<13t.
\end{equation}
The coarse bound confines $E$ to an interval on which the following
comparison is effective; it is not yet the bound needed for the
determinant argument. For a sharper estimate, use
\[
 S\leq t+\sqrt{2E},\qquad
 \sum_i d_i^4\leq(D-E)^2+(E-1)^2,
\]
and
\[
 \sum_i d_i\left(1-\frac{d_i^2}{D}\right)
 \leq\frac{E}{t}+\sqrt{2E}.
\]
Here the largest dimension contributes $d_3E/D\leq E/t$, and the
other two contribute at most $d_1+d_2\leq\sqrt{2E}$.
Combining these estimates with \eqref{eq:short-diagonal-lower} and
\eqref{eq:short-spectral-upper} gives $g(E)\leq-(2E+23)/t^2<0$, and in particular
\begin{equation}\label{eq:short-g}
 g(E):=2E-\frac{2E^2}{t^2}-\frac{2E}{t}
              -6\sqrt{2E}-4t-25\leq0.
\end{equation}
Suppose $E\geq E_0:=7t/2$. On $[E_0,13t]$,
\[
 g'(E)=2-\frac{4E}{t^2}-\frac2t-\frac{3\sqrt2}{\sqrt E}
 >2-\frac{54}{t}-\frac6{\sqrt{7t}}>0.
\]
But
\[
 (3t-57)^2-252t=9\bigl((t-60)(t-6)+1\bigr)>0,
\]
so
\[
 g(E_0)=3t-\frac{113}{2}-6\sqrt{7t}>\frac12.
\]
This contradicts \eqref{eq:short-g} and proves
\eqref{eq:short-dominance}. In particular $E<D/2$, so $M>D/2$ and
the largest simple is unique.

Write $X=X_3$ and let $N_X$ be its fusion matrix. Put
$y_j=\chi_j(X_3)^2/f_j$ for $1\leq j\leq3$.
Orthogonality gives $y_1+y_2+y_3=E/D$. The arithmetic--geometric mean
inequality now shows
\begin{align*}
 |\det N_X|^2
 &=M(f_1f_2f_3)y_1y_2y_3\\
 &<36D\left(\frac{E}{3D}\right)^3
 <\frac{343}{6\sqrt D}\leq\frac{343}{360}<1.
\end{align*}
Thus $\det N_X=0$, since it is an integer. This is an exact
integrality argument: the strict upper bound is already less than one
at $D=3600$, so no limiting or numerical assertion is involved. Finally,
$\chi_j(X_3)^2/f_j\leq E/D$ and $f_j<4$ imply
\eqref{eq:short-small-eigenvalues}.
\end{proof}

\subsection{The self-dual pseudo-unitary bound}\label{sec:pu}

\begin{proposition}\label{prop:pu}
If a pseudo-unitary fusion category has rank four and all its simple objects are self-dual, then its Frobenius--Perron dimension is less than $3600$.
\end{proposition}
\begin{proof}
Put $R=K_0(\mathcal C)$ and $D=\FPdim(\mathcal C)$, and suppose
$D\geq3600$.
Set $t=\sqrt D$, and let $X$ be the largest simple supplied by Lemma~\ref{lem:short-concentration}. Write the basis as $1,X,Y,Z$, with dimensions $1,x,y,z$, and put
\[
 K=\Q(x,y,z),\qquad r=[K:\Q],\qquad E=D-x^2.
\]

\smallskip\noindent
\emph{The rank of $N_X$ is $r$.}
The rational semisimple commutative algebra $R\otimes\Q$ is a product of number fields. Its character orbits are the embeddings of those fields, and the Frobenius--Perron orbit corresponds to $K$. If a character $\chi$ lies outside this orbit, each embedding of
$\Q(\chi(X))$ extends to its character field. Thus all conjugates of
$\chi(X)$ occur in the same non-Frobenius--Perron character orbit and
have absolute value less than one by \eqref{eq:short-small-eigenvalues}.
The norm of a nonzero algebraic integer has absolute value at least one,
so $\chi(X)=0$. On the Frobenius--Perron orbit the values are embeddings of the nonzero element $x\in K$, and are nonzero. Simultaneous diagonalization identifies these
character values, with their multiplicities, with the eigenvalues of
$N_X$. There are exactly $r$ embeddings in the Frobenius--Perron orbit,
so exactly $r$ eigenvalues are nonzero. Hence
\begin{equation}\label{eq:rank-degree}
 \operatorname{rank} N_X=r.
\end{equation}
The principal block of $N_X$ on $1,X$ is
$\left(\begin{smallmatrix}0&1\\1&a\end{smallmatrix}\right)$, so its rank is at least two. Its determinant vanishes by Lemma~\ref{lem:short-concentration}. Thus $r=2$ or $3$.

\smallskip\noindent
\emph{Rank two is impossible.}
Write
\[
 N_X=\begin{pmatrix}
 0&1&0&0\\
 1&a&b&c\\
 0&b&e&h\\
 0&c&h&k
 \end{pmatrix}.
\]
The upper left block is invertible and its Schur complement is exactly the lower right block. Rank two therefore gives $e=h=k=0$. Consequently $XY=bX$ and $XZ=cX$, so $y=b$ and $z=c$ are positive integers. The second character in the Frobenius--Perron orbit has formal codegree
\[
 1+(x')^2+b^2+c^2<4.
\]
It follows that $b=c=1$. Thus $Y$ and $Z$ are distinct self-dual invertible objects. Their product is a fourth invertible basis element, necessarily $X$: it cannot be $1,Y$, or $Z$ by cancellation. This gives $D=4$, a contradiction.

\smallskip\noindent
\emph{Rank three is impossible.}
Now $R\otimes\Q\cong\Q\times K$ with $[K:\Q]=3$. The remaining character $\rho$ is integer-valued and has formal codegree strictly between two and four, hence equal to three. Since $\rho(X)=0$, its values on $Y,Z$ are each $\pm1$. Orthogonality with the positive character excludes equal signs. Relabeling gives
\begin{equation}\label{eq:cubic-rational}
 \rho=(1,0,1,-1),\qquad z=y+1.
\end{equation}
Applying $\rho$ to products and using Frobenius reciprocity gives
nonnegative integers $b,c,s$ such that
\begin{align}
 XY&=bX+cY+cZ, &Y^2&=1+cX+sY+sZ,\label{eq:cubic-products}\\
 YZ&=cX+sY+(s+1)Z.\notag
\end{align}
The trace of $N_Y$ is $b+2s+1$. Subtracting $\rho(Y)=1$ gives
\begin{equation}\label{eq:cubic-trace}
 \Tr_{K/\Q}(y)=b+2s.
\end{equation}

By \eqref{eq:short-dominance},
\[
 E=1+y^2+(y+1)^2<\frac72t,
 \qquad x^2=t^2-E>(t-4)^2.
\]
The dimension equation for $Y^2$ gives $cx<y^2<E/2$, whence
\[
 c<\frac{7t}{4(t-4)}<2.
\]
Thus $c\leq1$. The equation for $XY$ gives
\[
 \delta:=y-b=\frac{c(2y+1)}x\geq0,\qquad
 \delta^2\leq\frac{2E-3}{D-E}<\frac7{t-7/2}<1.
\]
For either nonidentity embedding of $K$, write $u=y'$ and $v=x'$.
By \eqref{eq:short-small-eigenvalues}, $|v|<1/2$. Applying the embedding
to the equation for $Y^2$ gives
\[
 u^2=1+cv+s(2u+1).
\]
If $u\geq-1/2$, then $c\leq1$ and $s\geq0$ give $u^2>1/2$,
hence $u>1/\sqrt2$. Its formal codegree would satisfy
\[
 f'=1+v^2+u^2+(u+1)^2>3+\sqrt2>4,
\]
contrary to \eqref{eq:short-codegrees}. Thus $y',y''<-1/2$, and
\eqref{eq:cubic-trace} gives
\[
 2s=(y-b)+y'+y''<1-1=0,
\]
a contradiction. This completes the proof.
\end{proof}

\begin{proof}[Proof of Theorem~\ref{thm:uniform}]
Let $\mathcal C$ be a pivotal categorification of $R$ and put
$D=\FPdim(\mathcal C)$. If it has a non-self-dual simple,
Lemma~\ref{lem:nonselfdual} gives $D<45$. Otherwise,
Lemma~\ref{lem:exceptional} either gives $D<576$ or supplies a
pseudo-unitary Galois conjugate with the same fusion ring.
Proposition~\ref{prop:pu} then gives $D<3600$.
The preliminary bound $\mu(R)^2\leq D$ implies $\mu(R)<60$, hence
$\mu(R)\leq59$ since fusion coefficients are integers.
\end{proof}

%% file: source/criteria.tex
\section{The finite classification}\label{sec:established}
We classify the fusion rings of pivotal fusion categories of rank four over
$\C$. After Theorem~\ref{thm:uniform}, only a finite collection remains. We apply inexpensive structural and integer-arithmetic tests
before the character-field and induction calculations.

\subsection{Arithmetic and structural restrictions}
The first two restrictions below hold in arbitrary rank and require no
pivotal structure. The last restrictions concern the indicated rank-four
families; we spell out why their pseudo-unitary hypotheses cause no loss
for the pivotal classification.

\begin{proposition}\label{prop:published}
\begin{enumerate}
\item \emph{Formal-codegree $d$-numbers.}
Every formal codegree of a categorifiable fusion ring is a $d$-number
\cite[Theorem~1.2]{OstrikCodegrees}. An algebraic integer is a
$d$-number when its principal ideal in the ring of all algebraic
integers is Galois invariant. For minimal polynomial
$t^e+a_1t^{e-1}+\cdots+a_e$, this is equivalent to
$a_e^j\mid a_j^e$ for $1\leq j\leq e$
\cite[Lemma~2.7(v)]{OstrikCodegrees}.
\item \emph{Rank-two subring classification.} A rank-two fusion subring $\Z\{1,x\}$ with $x^2=1+ax$ is
categorifiable only for $a=0,1$ \cite[Main Theorem, p.~178]{OstrikRank2}.
\item \emph{Published rank-four families.}
If the simple Frobenius--Perron dimensions take exactly the values $1,d$
with $d>1$,
any categorification has a pseudo-unitary Galois conjugate
\cite[Lemma~3.1]{Schopieray}. Galois conjugation preserves all fusion coefficients. Applying the
pseudo-unitary theorems to that conjugate therefore yields necessary
conditions for any complex categorification, in particular a pivotal one:
\begin{enumerate}[label=(\roman*)]
\item The near-group fusion ring $\Near(C_3,m)$, with basis
$C_3\sqcup\{x\}$ and multiplication
$gx=xg=x$, $x^2=\sum_{g\in C_3}g+mx$, requires
$m\in\{0,2,3,6\}$ \cite[Theorem~1.1]{Larson}.
\item For the fusion ring with basis $\{1,g,x,y\}$ and
\[
 g^2=1,\quad gx=y,\quad x^*=y,\quad x^2=y^2=g+m(x+y),
\]
Larson's theorem requires $m\leq2$.
\item For the self-dual fusion ring with basis $\{1,g,x,y\}$ and
\[
 g^2=1,\quad gx=y,\quad x^2=1+ax+by,
\]
the pseudo-unitary classification in \cite[Theorem~2.4]{EIP} requires
\[
 (a,b)\in\{(0,0),(0,1),(1,0),(1,1),(2,2)\}.
\]
\end{enumerate}
In each of the first two nonpointed families the simple dimensions
take exactly the values $1,d$, and in the third $\FPdim(x)=\FPdim(y)=d$ because $y=gx$.
The pointed boundary cases are included directly. In particular,
\cite[Theorem~2.4]{EIP} is used only after this justified reduction;
we do not assume that every pivotal fusion category is pseudo-unitary.
\end{enumerate}
\end{proposition}

The filters are applied in the order recorded in Table~\ref{tab:filters}.

%% file: source/quadratic.tex
\subsection{A quadratic twist obstruction}\label{sec:quadratic}

The obstruction below concerns pivotal fusion categories over $\C$
whose Grothendieck ring belongs to an explicit self-dual
rank-four family. Its proof reduces to the pseudo-unitary case before
using positive dimensions.

\begin{lemma}\label{lem:quadratic-length}
If eighteen roots of unity have sum $-2-2\sqrt{10}$, the sum of their
squares has normalized rational trace $2$.
\end{lemma}
\begin{proof}
For a number in a finite extension $L/\Q$, normalized rational trace
means $[L:\Q]^{-1}\Tr_{L/\Q}$; it is unchanged on enlarging $L$.
Choose a cyclotomic field $L$ containing the roots and $K=\Q(\sqrt{10})$.
Let $T=[L:K]^{-1}\Tr_{L/K}$ and
$\ell(a+b\sqrt{10})=-a-8b$.
We claim that $\ell(T(\xi))\leq1$ for every root of unity $\xi$, with
equality possible only when $\xi=-1$ or $\xi$ is primitive of order $40$.

Write $m$ for the order of $\xi$. We denote the M\"obius function by
$\mu$ and Euler's totient function by $\varphi$; here $\mu$
is not to be confused with the fusion multiplicity $\mu(R)$.
Since $K$ has conductor $40$,
if $40\nmid m$ then $K\cap\Q(\xi)=\Q$, and
$T(\xi)=\mu(m)/\varphi(m)$. This gives the claim, with equality
only for $m=2$. If $m=40h$, the quadratic Gauss sum for the primitive
character of conductor $40$ gives
\[
 T(\xi)=0\quad\text{if }(40,h)>1,\qquad
 T(\xi)=\frac{\pm\mu(h)\sqrt{10}}{8\varphi(h)}
       \quad\text{if }(40,h)=1.
\]
Indeed, split the units modulo $m$ according to this quadratic character.
The untwisted sum is $\mu(m)=0$; the twisted sum vanishes on an extra
prime-power layer above $40$, and otherwise the Chinese remainder theorem
reduces it to $\pm\mu(h)\sqrt{40}$. Thus $|\ell(T(\xi))|\leq1$;
equality requires $(40,h)=1$ and $\varphi(h)=1$, hence $h=1$.

The value of $\ell$ on the prescribed sum is $18$, so equality holds
for each of the eighteen roots. A primitive fortieth root and its square
have normalized rational traces
$\mu(40)/\varphi(40)=\mu(20)/\varphi(20)=0$.
The original sum has rational trace $-2$, so exactly two roots are $-1$.
Their squares contribute $2$, proving the assertion.
\end{proof}

For $n\geq1$, define $R_n$ in the self-dual basis $(1,X,Y,Z)$ by putting
$W=X+Y+Z$ and
\[
\begin{array}{lll}
 X^2=1+nW-X-Y,&XY=nW-X,&XZ=nW,\\
 Y^2=1+nW,&YZ=nW+Z,&Z^2=1+nW+Y+Z.
\end{array}
\]
Its Frobenius--Perron dimension vector is
$(1,\delta-1,\delta,\delta+1)$, where
$\delta=(3n+\sqrt{9n^2+4})/2$; in particular
$\delta^2=3n\delta+1$ and $\FPdim(R_n)=6+9n\delta$.
The symmetric fusion matrices commute, as follows directly from the
displayed products; Lemma~\ref{lem:commuting-multiplication} therefore
gives associativity. Nonnegativity and
Frobenius reciprocity are also immediate from those products.

\begin{proposition}\label{prop:quadratic-twist}
Write $9n^2+4=t^2d$ with $t\geq1$ and $d>1$ squarefree. If $n=2$, or if
\begin{equation}\label{eq:quadratic-numerical-obstruction}
 \frac{nt}{2}\varphi(2d)>4n^2+2,
\end{equation}
then $R_n$ has no pivotal categorification over $\C$.
In particular, this excludes $R_n$ for $2\leq n\leq11$.
\end{proposition}
\begin{proof}
Suppose $\mathcal C$ is a pivotal categorification of $R_n$. Since
all its simple objects are self-dual, its pivotal structure is spherical. The four characters of $R_n$ are
\[
 (1,\delta-1,\delta,\delta+1),\quad
 (1,\delta'-1,\delta',\delta'+1),\quad
 (1,0,1,-1),\quad(1,1,-1,0),
 \qquad \delta'=-\delta^{-1}.
\]
The two rational characters vanish on a simple, so
Lemma~\ref{lem:galois-reduction} gives a pseudo-unitary Galois conjugate.
Use its canonical spherical structure and put $D=\FPdim(\mathcal C)$.
The formal codegrees are $D,D'=3+3(\delta')^2,3,3$; hence
\[
 \begin{gathered}
 I(1)=1\oplus A_1\oplus A_2\oplus A_3,\\
 \dim_{\mathcal Z(\mathcal C)}(A_1)=1+3n\delta,\qquad
 \dim_{\mathcal Z(\mathcal C)}(A_2)=\dim_{\mathcal Z(\mathcal C)}(A_3)=2+3n\delta,
 \end{gathered}
\]
with all four twists equal to one.

Use the induction matrix $M$ for the forgetful functor
$F:\mathcal Z(\mathcal C)\to\mathcal C$, introduced in
Section~\ref{sec:central}, and the Casimir matrix $B$ defined in
\eqref{eq:casimir}. By \eqref{eq:low-gram}, $M^{\mathsf T}M=B$.
Comparison of the coefficients of $1,\delta$ forces the rows for
$A_1,A_2,A_3$ to be
\[
 (1,a,3n-2a,a),\quad
 (1,b,3n-1-2b,b+1),\quad
 (1,c,3n-1-2c,c+1).
\]
The $X$-coordinate of the identity $L^{\mathsf T}\mathbf1=B e_0$
gives $a+b+c=3n-1$.
For $v=(0,1,-2,1)^{\mathsf T}$, direct multiplication gives
$v^{\mathsf T}Bv=18$; its values on these three rows are
$6(a-n),6(b-n)+3,6(c-n)+3$.
Their squared sum is at most $18$, forcing
$a=n$ and $\{b,c\}=\{n-1,n\}$.
Thus, up to interchanging $A_2,A_3$, their respective induction rows are
\[
 (1,n,n,n),\qquad (1,n-1,n+1,n),\qquad (1,n,n-1,n+1).
\]
Equality also forces $v$ to vanish on every remaining row.
These rows therefore have the form $(0,b-u,b,b+u)$,
where $b\geq|u|$ and $b>0$, and have dimension $3b\delta+2u$.
Subtracting the outer products of the four low rows from $B$ gives
\begin{equation}\label{eq:quadratic-gram}
 \sum b^2=6n^2+2,\qquad \sum bu=2n,\qquad \sum u^2=2.
\end{equation}
There are precisely two rows with $u\ne0$, each with $u=\pm1$.
Put $T_0=3n\delta+2$. The difference of the first twist-trace identities
for $Z$ and $X$, divided by two, reads
\begin{equation}\label{eq:quadratic-first-trace}
 \sum u(3b\delta+2u)\theta=-T_0.
\end{equation}
If the signs were opposite, write the two rows with parameters
$(b,u)=(p,1),(q,-1)$. Then $p-q=2n$ by
\eqref{eq:quadratic-gram}, so their positive dimensions differ by
$3(p-q)\delta+4=2T_0$. The left side of
\eqref{eq:quadratic-first-trace} would therefore have absolute value
at least $2T_0$, whereas its right side has absolute value $T_0$.
This contradicts the reverse triangle inequality. Both signs cannot be negative
by \eqref{eq:quadratic-gram}; hence they are positive.
Write their $b$ values as $p,q$, so $p+q=2n$, and their twists as
$\zeta,\eta$. Their dimensions $\alpha,\beta$ satisfy
$\alpha+\beta=2T_0$ and
$\alpha\beta=(9pq+4)\delta^2$, whereas
$T_0^2=(9n^2+4)\delta^2$. Equation~\eqref{eq:quadratic-first-trace} yields
\[
 \zeta\eta^{-1}+\eta\zeta^{-1}
 =2-\frac{3(9n^2+4)}{9pq+4}.
\]
This rational algebraic integer belongs to $[-2,-1]$.
The value $-2$ would imply $36pq=27n^2-4$, impossible modulo $9$.
The value $-1$ forces $p=q=n$ and $\zeta+\eta=-1$;
thus $\zeta,\eta$ are the primitive cube roots.

The other rows have the form $(0,b,b,b)$ with
$\sum b^2=4n^2+2$. Their twists consequently give a sum
$S=\sum b^2\theta$ of $4n^2+2$ roots of unity, with repetitions.
The first twist-trace identity for $X$ reduces to
\[
 0=n(1+3n\delta)+nT_0+3\delta S,
 \qquad
 S=n^2-n\delta=-\frac{n^2}{2}-\frac{nt}{2}\sqrt d.
\]
Duplicating these roots clears the denominators in $S$.
Larson's root-length theorem \cite[Theorem~6.2]{Larson} therefore gives
$4n^2+2\geq (nt/2)\varphi(2d)$, contradicting
\eqref{eq:quadratic-numerical-obstruction}. For $n=3,\ldots,11$ the
right-hand side of this lower bound is, respectively,
\[
 96,\ 144,\ 570,\ 480,\ 1232,\ 896,\ 3294,\ 2240,\ 6006,
\]
each strictly larger than $4n^2+2$.

For $n=2$, $S=-2-2\sqrt{10}$ is a sum of eighteen roots.
Lemma~\ref{lem:quadratic-length} says that their squared sum has
normalized rational trace $2$.
Since $X$ is self-dual, its second indicator is
$\epsilon\in\{1,-1\}$. The second twist-trace identity, using
$\zeta^2+\eta^2=-1$, gives for the squares of those same eighteen roots
\[
 S_2=S+\epsilon\frac{D}{3\delta}
     =-2+(2\epsilon-2)\sqrt{10}.
\]
Its normalized rational trace is $-2$ for either value of $\epsilon$,
a contradiction.
\end{proof}

In this family $6+27n^2<\FPdim(R_n)<9+27n^2$; consequently
$R_2,\ldots,R_{11}$ are precisely the parameters excluded by
Proposition~\ref{prop:quadratic-twist} in the window
$\FPdim(R_n)<3600$.

%% file: source/completion.tex
\subsection{Exact census and completion}\label{sec:completion}
Appendix~\ref{app:census} proves that the supplementary census contains
every rank-four fusion ring with $\FPdim(R)<3600$, exactly once. Its $8977$
sets of fusion data, the distribution by multiplicity, and the
certificate ledger are supplied in \texttt{supplementary}.
From the directory containing \texttt{supplementary}, the complete
verification command is
\begin{center}\small\texttt{python3 supplementary/check.py}\end{center}
In the arXiv source archive, this directory is \texttt{anc/}.
The command regenerates the census and certificates, checks their coverage and
all fusion-ring identities using exact arithmetic, and reconstructs the
fifteen sets of model fusion data. No external census is needed. The production
programs share some routines; separate checks use spectral-adjugate
characters, resultant norms, independent finite-field tests, and full
integer boxes for induction rows. Another program checks the
quadratic-family identities. The optional \texttt{--deep} flag also
removes the census pruning and checks an exhaustive small parameter box.
The supplementary guide states the dependencies and the precise scope
of each check; none is a proof-assistant certification of the
categorical arguments.

\input{source/generated_filters}

Every potentially spherical character orbit is tested against
nonvanishing, equality on duals, \eqref{eq:norm-test},
\eqref{eq:spherical-mass}, and formal-codegree field containment. Lemma~\ref{lem:residue-field}
provides finite-field certificates for failure of the last condition.
For an orbit whose global dimension is conjugate to $\FPdim(R)$,
Lemma~\ref{lem:galois-reduction} also permits the positive-dimension
criterion of Proposition~\ref{prop:low-test}. The enumeration exhausts
all required induction-row choices, up to permutations within
equal-formal-codegree blocks. An exclusion is accepted only when every
character orbit, and every required row choice, has an actual
violated necessary condition.

\begin{proof}[Proof of Theorem~\ref{thm:main}]
Theorem~\ref{thm:uniform} and Proposition~\ref{prop:census-coverage}
place every possible fusion ring in the census.
Lemma~\ref{lem:spherical} reduces pivotal existence to spherical
existence. The exact certificates apply
Proposition~\ref{prop:published}, Proposition~\ref{prop:quadratic-twist},
the spherical-character tests, and Proposition~\ref{prop:low-test},
leaving precisely the fifteen fusion rings in Table~\ref{tab:models}.
Each is identified by an explicit fusion-ring isomorphism with the fusion ring of its model; the cited constructions provide unitary, hence pivotal,
categorifications. This proves both directions.
\end{proof}

\begin{corollary}\label{cor:sharp}
For every pivotal fusion category $\mathcal C$ of rank four,
\[
 \FPdim(\mathcal C)\leq24+12\sqrt3,
 \qquad \mu\bigl(K_0(\mathcal C)\bigr)\leq6.
\]
Equality in either bound occurs precisely for the near-group fusion ring
of type $(C_3,6)$.
\end{corollary}
\begin{proof}
Apply Theorem~\ref{thm:main}. For type $(C_3,6)$, the noninvertible
simple has Frobenius--Perron dimension $d=3+2\sqrt3$, so
$\FPdim(\mathcal C)=3+d^2=24+12\sqrt3$ and
$\mu=6$. Every other row of Table~\ref{tab:models} has $\FPdim(\mathcal C)<44$ and
$\mu\leq3$. The radical comparisons are immediate. For $\delta_7$ and $\delta_E$,
the defining polynomials become polynomials with strictly positive
coefficients after substituting $x=44+u$; neither has a root at least
$44$. Finally $44<24+12\sqrt3$. The cited construction of the
near-group fusion category realizes both equalities.
\end{proof}

%% file: source/generated_filters.tex
% Generated by supplementary/check_manuscript.py; do not edit.
\begin{table}[ht]
\centering\small
\begin{tabular}{@{}lrr@{}}
\toprule
Stage & Excluded & Remaining \\
\midrule
Complete census & --- & 8977 \\
Rank-two fusion subring classification & 1676 & 7301 \\
Formal-codegree $d$-numbers & 6541 & 760 \\
Published rank-four families & 475 & 285 \\
Quadratic twist obstruction & 10 & 275 \\
Spherical-character arithmetic & 210 & 65 \\
Local central induction & 50 & 15 \\
\bottomrule
\end{tabular}
\caption{Successive exact filters. The last two stages cover every
potentially spherical character orbit and every required low-part choice.}
\label{tab:filters}
\end{table}

%% file: source/future_work.tex
\par\smallskip
The general, not necessarily pivotal, case is the subject of our work
in progress, \emph{Classifying Grothendieck rings of general fusion
categories of rank four}.

%% file: source/census_appendix.tex
\section{Completeness of the bounded census}\label{app:census}
This appendix concerns abstract rank-four fusion rings; no
categorifiability or ground-field assumption is imposed. It supplies the
mathematical coverage argument for the finite search. Checking that emitted
fusion data are valid and pairwise nonisomorphic does not by itself prove
that every fusion ring occurs. We therefore give the
complete presentations, the elimination branches, and the exact acceptance
condition. All parameters range over integers in $[0,M]$, where $M=59$;
write $\Lambda=3600$ for the exclusive dimension bound.

\subsection{Complete presentations}
By Lemma~\ref{lem:commutative}, every rank-four fusion ring is
commutative. Duality on its three nonunit basis elements is either the
identity or a transposition. Frobenius reciprocity gives the following two
presentations, with $N_0=I$ and our row convention
$(N_u)_{vw}=N_{uv}^{w}$.

In the self-dual case the parameters are $(a,b,c,d,e,f,g,h,i,j)$, and
\[
N_1=\begin{pmatrix}0&1&0&0\\1&a&b&c\\0&b&d&e\\0&c&e&f\end{pmatrix},\quad
N_2=\begin{pmatrix}0&0&1&0\\0&b&d&e\\1&d&g&h\\0&e&h&i\end{pmatrix},\quad
N_3=\begin{pmatrix}0&0&0&1\\0&c&e&f\\0&e&h&i\\1&f&i&j\end{pmatrix}.
\]
Their pairwise commutativity is equivalent to
\begin{align}
 ad-b^2+bg+ch-d^2-e^2+1&=0,\notag\\
 ae-bc+bh+ci-de-ef&=0,\notag\\
 be-cd+dh-eg+ei-fh&=0,\label{eq:sd-assoc}\\
 af+bi-c^2+cj-e^2-f^2+1&=0,\notag\\
 bf-ce+di-eh+ej-fi&=0,\notag\\
 df-e^2+gi-h^2+hj-i^2+1&=0.\notag
\end{align}
In the case with duality $(0)(1)(23)$ the parameters are $(a,b,c,d,e,f)$,
and
\[
 N_1=\begin{pmatrix}0&1&0&0\\1&a&b&b\\0&b&c&d\\0&b&d&c\end{pmatrix},\quad
 N_2=\begin{pmatrix}0&0&1&0\\0&b&c&d\\0&d&e&f\\1&c&e&e\end{pmatrix},\quad
 N_3=N_2^{\mathsf T}.
\]
Their commutativity is equivalent to
\begin{align}
 ac-b^2+2be-c^2-d^2+1&=0,\notag\\
 ad-b^2+be+bf-2cd&=0,\label{eq:ns-assoc}\\
 b(c-d)+d(e-f)&=0,\notag\\
 -c^2+d^2-e^2+f^2-1&=0.\notag
\end{align}
\begin{lemma}[Commuting multiplication operators]\label{lem:commuting-multiplication}
Let $k$ be a commutative ring and let $A$ be a $k$-module with a
$k$-bilinear multiplication and a two-sided identity $1$. For $x\in A$,
write $L_x(y)=xy$. The multiplication is associative and commutative
if and only if $L_xL_y=L_yL_x$ for all $x,y\in A$.
\end{lemma}
\begin{proof}
For an associative and commutative multiplication,
$L_xL_y=L_{xy}=L_{yx}=L_yL_x$. Conversely, suppose that the
multiplication operators commute. Applying $L_xL_y=L_yL_x$ to $1$
gives $xy=yx$. Hence, for all $x,y,z\in A$,
\[
 (xy)z=z(xy)=L_zL_x(y)=L_xL_z(y)=x(zy)=x(yz).
\]
Thus the multiplication is also associative.
\end{proof}

In the displayed presentations, the proposed product is bilinear on a
free abelian group and has identity $b_0$. The matrix of $L_{b_u}$
in the usual column convention is $N_u^{\mathsf T}$; consequently,
pairwise commutativity of the $N_u$ is equivalent to that of the
$L_{b_u}$, and the lemma proves associativity. The unit, duality,
integrality, nonnegativity, and Frobenius reciprocity axioms are already
built into the presentations.

\subsection{The self-dual branches}
Permutations of the three nonunit elements act transitively on the six
mixed coordinates $b,c,d,f,h,i$. Every isomorphism class therefore
has a representative with $b=\min(b,c,d,f,h,i)$. We enumerate all
such representatives before imposing the final isomorphism reduction.

\paragraph{\textbf{Positive minimal mixed coefficient.}}
Suppose $b>0$. Enumerate $b,c,d,e,f$, with $c,d,f\geq b$, and set
\begin{align*}
 H_0&=e(d+f)+bc,&G_0&=b^2+d^2+e^2-1,&J_0&=c^2+e^2+f^2-1,\\
 A_0&=e(b^2-c^2)+bc(f-d),&C_0&=e(bf-ce).
\end{align*}
Also put
\[
 B_0=b^3e-b^2cd+b(d-f)H_0-ebG_0+ecH_0.
\]
The first, second, and fourth equations in \eqref{eq:sd-assoc} give
\[
 h=\frac{H_0-ae-ci}{b},\qquad
 g=\frac{G_0-ad-ch}{b},\qquad
 j=\frac{J_0-af-bi}{c}.
\]
Substitution into the third equation and multiplication by $b^2$ give
\begin{equation}\label{eq:census-linear}
 C_0a+A_0i=-B_0.
\end{equation}
If $A_0\ne0$, put $q=\gcd(A_0,C_0)$. There are no solutions unless
$q\mid B_0$; otherwise $a$ lies in the unique residue class solving
\[
 (C_0/q)a\equiv-B_0/q\pmod{|A_0/q|},
\]
and $i$ is determined by \eqref{eq:census-linear}. A modulus of one
imposes no restriction. If $A_0=0\ne C_0$, determine $a$ and
retain every bounded $i$. If $A_0=C_0=0$, require $B_0=0$ and
retain every bounded pair $(a,i)$. In all cases retain only integral
$g,h,j$, with $a,g,j\in[0,M]$ and $h,i\in[b,M]$, and test all
six original equations.

The implementation discards outer tuples only by necessary interval
conditions. Besides $b,c,d,f\in[b,M]$ and $e\in[0,M]$, these are
\begin{gather*}
 b(b+c)\leq H_0\leq M(e+b+c),\qquad
 bc\leq G_0\leq M(d+b+c),\\
 b^2\leq J_0\leq M(f+b+c),\\
 \min(0,C_0M)+\min(A_0b,A_0M)\leq-B_0
 \leq\max(0,C_0M)+\max(A_0b,A_0M).
\end{gather*}
The first three follow from the displayed equations for $h,g,j$;
the last bounds the left side of \eqref{eq:census-linear} over its
allowed rectangle. The further necessary dimension inequality
\[
 2+2b^2+2c^2+d^2+2e^2+f^2<\Lambda
\]
is the bound on $B_{11}$ after discarding its nonnegative summand $a^2$.
Here $B=\sum_u N_uN_u^{\mathsf T}$ is the Casimir matrix.

\paragraph{\textbf{One zero mixed coefficient.}}
Suppose $b=0<c$. The equations imply
\[
 h=\frac{d^2+e^2-ad-1}{c},\qquad
 i=\frac{e(d+f-a)}{c},\qquad
 j=\frac{c^2+e^2+f^2-af-1}{c},
\]
and $eg=h(d-f)+ei-cd$.
For $e>0$, enumerate $c,d,e,f$ and the progression
$a\equiv d+f\pmod{c/\gcd(c,e)}$. The bounds on $h,i,j$ give
necessary bounds on $a$; exact division determines $g$. Test all
six equations and all parameter bounds.

For $e=0$, one has $i=0$, and the first and third equations imply
$d,h>0$ and $\gcd(c,d)=1$. Indeed $d=0$ would give $ch=-1$,
and $h(d-f)=cd>0$ forces $h>0$; divisibility of $d(d-a)-1$
by $c$ gives coprimality. Enumerate $c,d$ and
$a\equiv d-d^{-1}\pmod c$, with the congruence unrestricted for
$c=1$. Determine $h$, put $f=d-cd/h$, and determine $j$ by the
displayed formula. Retain every bounded $g$, testing exact integrality,
all bounds, and all six equations.

\paragraph{\textbf{Two zero mixed coefficients, with $e>0$.}}
For $b=c=0<e$, the first, second, and fourth equations yield
$a=d+f$ and $df=e^2-1$. Enumerate every bounded factor pair,
including the zero-factor possibilities when $e=1$. For every
bounded $h,i$, determine $g,j$ from
\[
 eg=h(d-f)+ei,\qquad ej=eh+i(f-d).
\]
Retain the integral bounded solutions satisfying all six equations.

\paragraph{\textbf{The remaining zero branch.}}
For $b=c=e=0$, the equations force $a=0$ and $d=f=1$ and leave
\[
 gi+hj=h^2+i^2-2.
\]
For $h>0$, enumerate $h,i,g$ and determine $j$; for $h=0<i$,
determine $g$ and retain every bounded $j$. The case $h=i=0$
is impossible. As before, test all original equations and bounds.

These four branches are exhaustive. In each division step a nonzero
denominator has been explicitly assumed; when a coefficient vanishes,
every remaining bounded solution is retained. Finally keep the
lexicographically least tuple among its permutations preserving the
minimal value of $b$. Each orbit has exactly one such tuple. Thus the
self-dual enumeration has neither an omitted zero case nor a duplicate
fusion-ring isomorphism class.

\subsection{The branch with a dual pair}
Enumerate $c,d,e$. The last equation in \eqref{eq:ns-assoc} determines
the unique possible nonnegative integer
$f=\sqrt{1+c^2-d^2+e^2}$; test its square and bound exactly.
If $c\ne d$, the third equation determines
$b=-d(e-f)/(c-d)$. If $c=d$, require $d(e-f)=0$ and retain
every bounded $b$. Determine $a$ from the first equation when $c>0$,
from the second when $c=0<d$, and retain every bounded $a$ when
$c=d=0$. Test all four original equations and bounds.
The only unit-preserving relabelling compatible with this duality
exchanges the dual pair. It leaves the six parameters unchanged, so
these tuples already represent distinct isomorphism classes.

\subsection{The dimension test and the coverage conclusion}
The positive dimension vector is an eigenvector of the symmetric
nonnegative Casimir matrix $B$, with eigenvalue $\FPdim(R)$, which is therefore its largest eigenvalue
by Lemma~\ref{lem:commutative}. Therefore
\begin{equation}\label{eq:census-pd}
 \FPdim(R)<\Lambda\quad\Longleftrightarrow\quad
 \Lambda I-B\text{ is positive definite}
 \quad\Longleftrightarrow\quad
 \det(\Lambda I-B)_{[1,k]}>0\ (1\leq k\leq4).
\end{equation}
Here $[1,k]$ denotes the leading principal $k$-by-$k$ submatrix.
The final equivalence is Sylvester's criterion. All determinants are
integers; a zero determinant is rejected, so equality at the boundary
is never admitted. The C++ enumeration uses fraction-free elimination,
and the Python fusion-data checker recomputes the same quantities with
arbitrary-precision integers.

\begin{proposition}\label{prop:census-coverage}
The preceding procedure with $M=59$ and $\Lambda=3600$ emits exactly
one representative of every rank-four fusion ring with $\FPdim(R)<3600$.
\end{proposition}
\begin{proof}
Every rank-four fusion ring has one of the two complete presentations.
All solutions of their associativity equations in the parameter box
occur in the enumerated branches, and the isomorphism reduction retains
exactly one representative. The acceptance condition
\eqref{eq:census-pd} is equivalent to the desired strict dimension
bound. Conversely, every fusion ring below that bound has
$\mu(R)^2\leq\FPdim(R)$
by Lemma~\ref{lem:commutative}, hence $\mu(R)\leq59$, and therefore
lies in the searched box. This proves both directions.
\end{proof}

%% file: source/closing.tex
\subsection*{Acknowledgments}
S\'ebastien Palcoux is supported by the National Natural Science
Foundation of China (NSFC), Grant no.~12471031.
Arnaud Plessis is supported by the Beijing Municipal Natural Science
Foundation (No.~IS25018).
GPT-6 Astra was used as a research assistant. The authors take full
responsibility for the final manuscript.

%% file: source/bibliography.tex
\raggedbottom